\documentclass[review,onefignum,onetabnum]{siamart171218}

\usepackage{lipsum}
\usepackage{amsfonts}
\usepackage{graphicx}
\usepackage{epstopdf}
\usepackage{algorithmic}
\ifpdf
  \DeclareGraphicsExtensions{.eps,.pdf,.png,.jpg}
\else
  \DeclareGraphicsExtensions{.eps}
\fi

\newsiamremark{remark}{Remark}
\newsiamremark{hypothesis}{Hypothesis}
\crefname{hypothesis}{Hypothesis}{Hypotheses}
\newsiamthm{claim}{Claim}

\headers{Perturbation Theory in Tangent Space}{F. Hrobat, Y. Nakatsukasa}

\title{
Matrix Perturbation Theory in the Tangent Space of Isospectral Matrices
\thanks{Version of \today.
\funding{YN is supported by EPSRC grants EP/Y010086/1 and EP/Y030990/1. FH is supported by Mathematical Institute Scholarship.}}}

\author{Francesco Hrobat\thanks{Mathematical Institute, University of Oxford, OX2 6GG, UK
  (\email{francesco.hrobat@maths.ox.ac.uk}, 
  \email{yuji.nakatsukasa@maths.ox.ac.uk}).}
\and Yuji Nakatsukasa\samethanks
}

\usepackage{amsopn}

\makeatletter
\newcommand*{\addFileDependency}[1]{
  \typeout{(#1)}
  \@addtofilelist{#1}
  \IfFileExists{#1}{}{\typeout{No file #1.}}
}
\makeatother

\ifpdf
\hypersetup{
  pdftitle={
Matrix Perturbation Theory in the Tangent Space of Isospectral Matrices
  },
  pdfauthor={Francesco Hrobat and Yuji Nakatsukasa}
}
\fi

\usepackage{amsmath}
\usepackage{amssymb}
\usepackage{comment}
\usepackage{framed}
\usepackage{mathrsfs}
\usepackage{tikz}
\usepackage{pgfplots}
\pgfplotsset{compat=1.18}
\usepgfplotslibrary{groupplots}
\usepackage{float}

\newcommand{\ignore}[1]{}
\newcommand{\C}{\mathbb{C}}
\newcommand{\R}{\mathbb{R}}
\newcommand*\samethanks[1][\value{footnote}]{\footnotemark[#1]}

\begin{document}

\maketitle

\begin{abstract}
Eigenvalue and eigenvector perturbation theory is a fundamental topic in several disciplines, including numerical linear algebra, quantum physics, and related fields. The central problem is to understand how the eigenvalues and eigenvectors of a matrix $A \in \mathbb{C}^{n \times n}$ change under the addition of a perturbation matrix $E \in \mathbb{C}^{n \times n}$.
Much of the existing literature focuses on structured perturbations. For example, in 
[C.-K. Li and R.-C. Li, Linear Algebra Appl. 2005], 
the matrix $A$ is assumed to be Hermitian and block diagonal, while the perturbation $E$ is Hermitian and block off-diagonal. In this work, we investigate a different structured 
setting in which the perturbation has the commutator form $E = AB - BA$ for some matrix $B$, which we show to be a generalization of the block diagonal structure considered by Li and Li. 
First, we extend their main result by showing that the perturbation of the $i$-th eigenvalue of $A$, denoted by $\lambda_i$, is of order $\|E\|^2 / \eta_i$, where $\eta_i = \min_{j \neq i} |\lambda_i - \lambda_j|$ is the spectral gap associated with $\lambda_i$.
Second, we provide a detailed analysis of the role played by the matrix $B$ in the perturbation of the eigenvectors. This analysis is further generalized to the case of block-diagonal matrices with multiple eigenvalues, as well as to perturbed singular values and eigenvalues of Jordan blocks.

\end{abstract}


\section{Introduction}

Matrix perturbation theory is a classical and central topic in a wide range of fields, including numerical linear algebra \cite{Haa2025}, quantum physics \cite{Das2002} and dynamical systems \cite{Ash2012}. A fundamental problem concerns the stability of spectral quantities under perturbations: given a matrix
$A \in \mathbb{C}^{n \times n}$ and a perturbation
$E \in \mathbb{C}^{n \times n}$, how do the eigenvalues and eigenvectors of $A+E$ deviate from those of $A$?

For eigenvalues, an extensive theory is available. Classical results such as Weyl’s theorem \cite{Wey1912}, the Bauer-Fike theorem \cite{Bau1960}, and related bounds \cite[\S IV]{Ste1990} \cite[\S III]{Bha2012} show that, for a normal matrix $A$,
\[
|\lambda_i(A+E) - \lambda_i(A)| \le \|E\|,
\]
where $\|\cdot\|$ denotes the spectral norm. This bound is sharp in general and cannot be improved without additional assumptions, as is already evident in the scalar case.

Such bounds are simple and only depend on the norm of the perturbation. The major drawback of using these bounds is that they do not exploit possible structural properties of $E$. For this reason, numerous works have analyzed the case of structured $E$ and better bounds had been produced. 
A notable example is provided by Li and Li \cite{Li2005}, who consider Hermitian matrices of the form
\[
A =
\begin{bmatrix}
A_1 & 0 \\
0   & A_2
\end{bmatrix},
\qquad
E =
\begin{bmatrix}
0 & \widetilde{E} \\
\widetilde{E}^* & 0
\end{bmatrix}.
\]
In this setting, eigenvalue perturbations are shown to be of second order in $\|E\|$, namely
\[
|\lambda_i(A+E) - \lambda_i(A)| \le \frac{\|E\|^2}{\eta},
\]
where
\[
\eta = \min_{\lambda \in \Lambda(A_1),\ \mu \in \Lambda(A_2)} |\lambda - \mu|
\]
denotes the spectral gap between the two diagonal blocks. This result highlights the crucial role of spectral separation and shows that structured perturbations may lead to substantially improvement in eigenvalue stability. This principle has also been exploited in subsequent works \cite{Haa2025, Laz2025}.

In contrast, analogous results for eigenvectors are considerably more delicate \cite{Dav1970,stewart2}. Eigenvectors are inherently less stable, particularly when the spectral gap is comparable to the perturbation magnitude. For instance, consider the diagonal matrix
\[
\begin{bmatrix}
1 & 0 \\
0 & 1+\varepsilon
\end{bmatrix}
\]
where $\varepsilon$ is thought as small, and its perturbed counterpart
\[
\begin{bmatrix}
1 & \varepsilon \\
\varepsilon & 1+\varepsilon
\end{bmatrix}.
\]
Although the perturbation has norm $\mathcal{O}(\varepsilon)$ and the eigenvalues deviate by the same order, the associated eigenvectors rotate by an $\mathcal{O}(1)$ angle.
Despite this difficulty, eigenvector information is often more powerful than eigenvalue estimates. Indeed, approximations of left and right eigenvectors yield accurate eigenvalue estimates via the Rayleigh quotient \cite{Bha2012}.

In this work, we provide an extensive analysis on perturbation of eigenvalues and eigenvectors when the matrix $E$ has the particular structure
\[ E = AB-BA\]
where $B \in \C^{n \times n}$. In other words, $E$ has the form of a \textit{commutator}
\[ E = [A,B] := AB-BA\]
involving the matrix $A$ itself and another matrix $B$.
By considering this structure, we have a twofold objective: first, we generalize the result by Li and Li \cite{Li2005} to this kind of perturbation, which is a more general structure than block off-diagonal (see \Cref{sec:2x2}).

Secondly, we show the important role that the matrix $B$ plays in the perturbation theory of eigenvalues and, more importantly, of eigenvectors, providing a strong connection with classical results. We also show how some existing quantities in perturbation theory greatly simplify if written in terms of the matrix $B$. This can already be appreciated if we look at the asymptotic expansions (see also \Cref{sub:preli}): letting $x_i,y_i$ denote the right and left eigenvectors of the matrix $A$ corresponding to the eigenvalue $\lambda_i(A)$, we show that the asymptotic expansion of the perturbed eigenvalue $\lambda_i(A+E)$ is

\[ \lambda_i(A+E) = \lambda_i(A) + y_i^*BEx_i - y_i^*BEBx_i + \mathcal{O}(\|E\|^4)\]
where 
\[ \eta_i = \min_{i \neq j}|\lambda_i(A) - \lambda_j(A)|\]
is the spectral gap relative to the $i$-th eigenvalue.
Even simpler, the asymptotic expansions of the perturbed eigenvectors $x_i(A+E)$ and $y_i(A+E)$, respectively the perturbation of $x_i$ and $y_i$, are 

\[ x_i(A+E) = x_i -Bx_i+\mathcal{O}(\|E\|^2) \quad \text{and} \quad y_i(A+E) = y_i +B^*y_i+\mathcal{O}(\|E\|^2).\]
Note that these asymptotic expansions do not require any information about eigenvalues or eigenvectors different from $y_i$, $x_i$ or $\lambda_i(A)$, because this information is inherently encoded inside the matrix $B$. Besides asymptotic expansions, we provide non-asymptotic bounds, often more useful in fields such as numerical linear algebra.

Not all matrices can be written as a commutator: in fact, given the matrix $E$, the matrix $B$ exists if and only if $E$ belongs to the \textit{tangent space of the isospectral manifold} to which $A$ belongs. 
This idea appears in previous works dealing with isospectral flows. One foundational work is the one by Lax \cite{Lax1968} in the setting of matrix ODEs of the form

\begin{equation}\label{eq:laxeq}
\frac{d}{dt}A(t) = [A(t),B(t)].
\end{equation}
In this setting, $B(t)$ plays the role of our matrix $B$. 
One of the main results of Lax's paper is that $A(t)$ is isospectral to $A(0)$ for all $t$.
This can be clearly seen if $B(t) \equiv B$ is constant as
\[ A(t) = e^{-tB}A(0)e^{tB}.\]

Isospectral flows are employed in different settings: the continuous version of $QR$ algorithm, the \textit{$QR$ flow} \cite{Chu2008} takes the form of
\[ \dot{A} = [A, B(A)]\]
where
\[ B(A) = A_{-} - A_{-}^*\]
where $A_{-}$ is the strictly lower triangular part of $A$. The Brockett-Wegner flow \cite[\S 7]{Chu2008} 
\[ \dot{A} = [A,[A,N]]\]
where $N$ is diagonal. Additionally, Lax flows appear in Riemannian optimization over the isospectral manifold \cite{Chu1990}.
On top of these, isospectral flows appear also in physics, for example in Heisenberg equations of motions \cite[\S 2.2.3]{Sak2021}.

Equations of the form \Cref{eq:laxeq} where $A(t)$ is Hermitian are usually solved considering the Lie group splitting \cite{Wat1984}
\begin{align*}
\begin{split}
\mathcal{M} &\iff U(n) \\
A(t) &\iff Q(t)
\end{split}
\end{align*}
where $\mathcal{M}$ is the isospectral manifold and $A(t) = Q(t)D_0Q(t)^*$ is the eigenvalue decomposition. In this case, $Q(t)$ satisfies

\[ \dot{Q}(t) = -B(t)Q(t).\]
This shows the close connection between the matrix $B$ and the perturbed eigenvector matrix.

The remainder of this paper is organized as follows: in \Cref{sec:descrtang} we describe the tangent space of isospectral matrices. In \Cref{sec:eigen} we provide the main results and the non-asymptotic bounds.
To have a complete picture, in \Cref{sec:gener} we will also analyse the case where $E$ is not on the tangent by splitting $E$ as the sum of a matrix in the tangent and another matrix orthogonal to the tangent. In this section, we show that we can construct the matrix $B$ by looking at the projection of the perturbation onto the tangent and how $B$ also plays an important role in this case.
In \Cref{sec:2x2} we deal with the block-diagonal case, which explicitly includes the Li and Li setting. This is needed if the eigenvalues are not simple or if we have clusters of eigenvalues.
In \Cref{sec:hermi} we refine previous results under the hypothesis of Hermitian matrices.
In \Cref{sec:svd} we transpose the same results for the SVD. In this second setting, the manifold will consist of matrices with the same singular values. This manifold has no particular widely-adopted name. The characterization of the tangent space is the following: $E$ belongs to the tangent space, if and only if $E$ can be written as
\[ E = S_1A + AS_2\]
where $S_1,S_2$ are skew-symmetric matrices of suitable dimensions ($A$ may be rectangular).
Finally, in \Cref{sec:jordan} how this setting can be extended to Jordan blocks.

\section{Description of the tangent space of isospectral matrices}\label{sec:descrtang}
Given a diagonalizable matrix $A \in \C^{n \times n}$, we are interested in studying the perturbation of the eigenvalues under the additive perturbation $E \in \C^{n \times n}$ in the tangent space of the isospectral manifold. First of all, we provide an informal argument to justify that $E$ is in the tangent if and only if 
\[ E = AB-BA.\]
A rigorous proof for the general setting of Lie groups can be found in \cite{Hal2015}.

Unless otherwise stated, we will always assume that all eigenvalues of $A$ are simple. Consider the eigenvalue decomposition $A =XDX^{-1}$. The isospectral manifold can be parametrized as 
\begin{align}\label{eq:manifoldparam}
\begin{split}
GL(n, \C)/D^{\times} &\longrightarrow \C^{n \times n}\\
 W &\longrightarrow WDW^{-1}
\end{split} 
\end{align}
where $GL(n, \C)/D^{\times}$ is the quotient of the space of invertible matrices by the following equivalence relation
\[ W_1 \sim W_2 \text{\quad if and only if \quad} W_1 = W_2D\]
where $D$ is an invertible diagonal matrix. This is because the eigenvector matrix is unique up to rescaling the eigenvectors. 
If we are considering only Hermitian matrices, we consider the parametrization
\begin{align*}
\begin{split}
U(n)/T^n &\longrightarrow \C^{n \times n}\\
 Q &\longrightarrow QDQ^*
\end{split} 
\end{align*}
where $U(n)/T^n$ is the quotient of the unitary matrices by unitary diagonal matrices. 

In order to identify the tangent space, we apply a small perturbation $P$ to $W$
\[ W+P \longrightarrow (W+P)D(W+P)^{-1}.\]
Assuming that $P$ is small enough, we can expand $(W+P)D(W+P)^{-1}$ to the first order 
\begin{align*}
&(W+P)D(W+P)^{-1} = W(I+W^{-1}P)D(I +W^{-1}P)^{-1}W^{-1} \\
&\approx W(I +W^{-1}P)D(I-W^{-1}P)W^{-1}  = A + A(-PW^{-1}) - (-PW^{-1})A.    
\end{align*}
As $P$ is generic, it follows that the tangent space is parametrized by
\[ \C^{n \times n} \longrightarrow \C^{n \times n}\]
\[ B \longrightarrow AB-BA = [A,B].\]
Furthermore, we have already seen that this implies that the perturbation of the eigenvector matrix is equal to
\[ P = -BW.\]
Note that the matrix $WDW^{-1}$ in \Cref{eq:manifoldparam} is identified by $n^2-n$ parameters, as the column of $W$ can be rescaled (i.e., the degrees of freedom of $GL(n, \C)/D^{\times}$ are $n^2 -n$). This means that the matrix $AB-BA$ is also identified by $n^2-n$ parameters. Define $x_i$ as the $i$-th column of $X$ and $y_i$ the $i$-th column of $X^{-*}$. It holds that
\[ y_i^*(AB-BA)x_i = e_i^*DXBx_i - y_i^*BX^{-1}De_i = \lambda_i(e_i^*XBX^{-1}e_i- e_i^*XBX^{-1}e_i) = 0.\]
As expected, we have $n$ constraints. Furthermore, we have the following characterization of the tangent space

\begin{theorem}[Another characterization of the tangent space]\label{thm:chartang}

Let $A \in \C^{n \times n}$ be a diagonalizable matrix. Let $\lambda_i$ be the $i$-th eigenvalues of $A$ and let $x_i,y_i \in \C^{n}$ the right and left eigenvectors relative to $\lambda_i$. Then $E \in \C^{n \times n}$ belongs to the tangent space of the isospectral manifold if and only if

\[ y_i^*Ex_i = 0, \qquad i= 1,\dots ,n.\]
    
\end{theorem}
These $n$ constraints can be alternatively derived by noting that if $B = p(A)$ is a polynomial in $A$, then
\[ Ap(A) - p(A)A = 0.\]
These constraints are unsurprising: in fact, according to classic perturbation theory \cite{Ste1990,Gre2020}, the first order perturbation of $i$-th eigenvalue is equal to

\[ \lambda_i(A+E) - \lambda_i(A) \approx  y_i^*Ex_i.\]
Now we consider the particular case where the matrix $A$ is block-diagonal
\[ A = \begin{bmatrix}
    A_1 & 0 \\
    0 & A_2
\end{bmatrix}.\]
Any block off-diagonal perturbation
\[ E = \begin{bmatrix}
    0 & E_1 \\
    E_2 & 0
\end{bmatrix}\]
satisfies the $n$ constraints, which means that we can find a matrix $B$ such that 
\[ E = AB-BA.\]
In fact, if we consider the right and left eigenvectors associated with an eigenvalue of the first block, they are of the form
\[ x_i = \begin{bmatrix}
    \Tilde{x}_i \\
    0
\end{bmatrix} \qquad y_i = \begin{bmatrix}
    \Tilde{y}_i\\
    0
\end{bmatrix}\]
where the $0$ vector has the same length as the dimension of $A_2$. This implies that
\[ y_i^*Ex_i = 0.\]
More explicitly, the matrix $B$ is equal to (up to a polynomial in $A$)
\[ B = \begin{bmatrix}
    0 & B_1 \\
    B_2 & 0
\end{bmatrix}\]
where $B_1$ and $B_2$ satisfy
\[ A_1B_1-B_1A_2 = E_1 \qquad A_2B_2 -B_2A_1 = E_2.\]
Therefore, the setting we are studying here is a generalization of what has been studied in previous works \cite{Li2005, Mat1998}.
Note that, in the block case, we can still compute $B_1$ and $B_2$ even if the eigenvalues of $A_1$ and $A_2$ are not simple. All we require is the separation between the spectrum of $A_1$ and that of $A_2$. 

Without loss of generality (aside from the diagonalizability assumption), we will consider the case where the matrix $A$ is diagonal (and so it will be denoted by $D$). In fact, if $A = XDX^{-1}$, the matrix $A+E$ can be rewritten as
\[ A + E = XDX^{-1} + E = X(D + X^{-1}EX)X^{-1}.\]
If $E$ is equal to $AB-BA$, then
\[A +E = X(D + X^{-1}ABX - X^{-1}BAX)X^{-1} = X(D + DX^{-1}BX - X^{-1}BXD)X^{-1}\]
and we reduce to the diagonal case where the matrix $B$ is now $X^{-1}BX$.
All the results in the following sections will hold for generic diagonalizable matrices if we make the substitutions
\[ B \leftarrow X^{-1}BX, \quad E \leftarrow X^{-1}EX\]
and the subsequent estimates
\[ \|E\| \leftarrow \kappa_2(X)\|E\|, \quad \|B\| \leftarrow \kappa_2(X)\|B\|.\]

\section{Bounds for eigenvectors and eigenvalues}\label{sec:eigen}

We start by listing some results that will be useful later on.

\subsection{Preliminary computations and asymptotic expansions}\label{sub:preli}

In this section, we will assume that the eigenvalues of the matrix we are considering are simple.  We consider the block case in \Cref{sec:2x2}. Furthermore we indicate with $D$ a diagonal matrix, and with $A$ a generic matrix.

\begin{lemma}\label{lem:zerodiag}
Consider a diagonal matrix $D$ with distinct eigenvalues $\{\lambda_i\}_{1}^n$ on the diagonal. A perturbation $E$ belongs to the isospectral manifold of $D$ if and only if 
\[ diag(E) = 0.\]
\end{lemma}
\begin{proof}
    Let $e_i$ the $i$-th vector of the canonical basis. According to \Cref{thm:chartang}, the matrix $E$ belongs to the isospectral manifold if and only if 
    \[ e_i^*Ee_i = 0\]
    for all $i$, as $e_i$ is the left and right eigenvector associated with $\lambda_i$.
\end{proof}

 Given a matrix $E$ in the tangent space, we can compute the matrix $B$ such that $E = AB-BA$ using the following lemma.

\begin{lemma}\label{lem:compB}
Consider a diagonal matrix $D$ with distinct eigenvalues $\{\lambda_i\}_{1}^n$ on the diagonal, and a  perturbation $E$ with $diag(E) = 0$. Then there exists a matrix $B$ such that
\[ E = DB-BD\]
and the entry $B(i,j),\ i \neq j$  is equal to
\[ B(i,j) = \frac{E_{i,j}}{\lambda_i-\lambda_j} \quad i \neq j.\]
The entries on the diagonal are free (as we have $n$ degrees of freedom). We will impose the diagonal of $B$ to be equal to $0$ for simplicity.
Equivalently,
\begin{equation}\label{eq:Bi}
B(:,i) = (D-\lambda_iI)^{\dagger}E(:,i)    
\end{equation}

\end{lemma}

In the generic case, when $A$ is not diagonal, we can reformulate \Cref{lem:compB} in the following way.

\begin{lemma}\label{lem:groupinv}
    Consider a matrix $A = XDX^{-1} \in \C^{n \times n}$ where $D$ is diagonal with distinct eigenvalues. Let $E$ be perturbation in the tangent space, then
    \[ E = AB-BA\]
    and
    \[ Bx_i = (A-\lambda_iI)^{\#}Ex_i\]
    where $(\cdot{})^\#$ is the group inverse \cite{Mey1988,Gre2020}, \cite[pp. 240-241]{Ste1990}, and $x_i$ denotes the $i$-th column of $X$. 
\end{lemma}

\Cref{lem:groupinv} links our setting to the classical setting of eigenvector perturbation \cite{Mey1988}.
The first result we present is the asymptotic expansions of the eigenvalues and eigenvectors.
These are known expansions, especially in the physics perturbation theory community and these asymptotic expansions have already been derived by Landau et al. \cite[\S VI]{Lan1958}.
To have clearer results, we introduce the following definition.

\begin{definition}\label{def:t}
Let $t$ be a parameter in $[0,1]$. Define $x_i(t)$ as the right eigenvector of $D+tE$ associated with the eigenvalue $\lambda_i(t)$, which is the perturbation of the $i$-th eigenvalue of $D$, denoted as $\lambda_i := \lambda_i(0)$. Similarly, denote by $y_i(t)$ the left eigenvector associated with the same eigenvalue $\lambda_i(t).$
\end{definition}

\begin{lemma}\label{lem:asympvec}
    Using the definition in \Cref{def:t}, the asymptotic expansion of $x_i(t)$ is
    \[ x_i(t) = e_i - tB(:,i) + t^2(D-\lambda_iI)^{\dagger}EB(:,i) + \mathcal{O}((t\|E\|)^3)\]
    assuming that $B$ has zero diagonal.
\end{lemma}

\begin{proof}
Recalling that eigenvalues and eigenvectors of a simple eigenvalue are analytic functions of the matrix entries~\cite{Kat1995},
the result can be recovered by considering
\[ x_i(t) = e_i + tx^{(1)}_i + t^2x_i^{(2)} + \dots\]
\[ \lambda_i(t) = \lambda_i + t\lambda_i^{(1)} + t^2\lambda_i^{(2)} + \dots\]
where $t$ ranges between $0$ and $1$ and then match the terms with the same power of $t$ in the equality

\[ (D+tE)(e_i + tx^{(1)}_i + t^2x_i^{(2)} + \dots) = (\lambda_i + t\lambda_i^{(1)} + t^2\lambda_i^{(2)} + \dots)(e_i + tx^{(1)}_i + t^2x_i^{(2)} + \dots).\]
    
\end{proof}

A similar result can be derived on the left eigenvector, $y_i(t)$, as it is the right eigenvector of $D^*+tE^*$ and $-B^*$ satisfies
\[ D^*(-B^*) - (-B^*)D^* = E^*.\]
In particular, we have 
\[ y_i^*(t) = e_i + tB(i,:)- t^2B(i,:)E(D-\lambda_iI)^{\dagger} + \mathcal{O}((t\|E\|)^3).\]
\Cref{lem:asympvec} implies that the whole eigenvector matrix $X(t)$ of $D+tE$ is well approximated by
\[ I -tB.\]
The same hold for the inverse $X(t)^{-1}$, which is well approximated by
\[ I + tB.\]
This is not surprising, as we have derived the tangent space using the identity

\[ (I+tB)D(I+tB)^{-1} \approx D + t(DB-BD).\]
Importantly, from the asymptotic expansion of eigenvectors, we are able to recover an asymptotic expansion of a greater order for the eigenvalues.

\begin{lemma}

The asymptotic expansion of $\lambda_i(t)$ is

\[ \lambda_i(t) = \lambda_i + t^2B(i,:)E(:,i) - t^3B(i,:)EB(:,i) + \mathcal{O}((t\|E\|)^4).\]
    
\end{lemma}

\begin{proof}
    This result follows from the Rayleigh-Ritz quotient identity
    \[ \lambda_i(t) = \frac{y_i(t)^* (D+tE)x_i(t)}{y_i(t)^*x_i(t)}\]
    and using the asymptotic expansions of the eigenvectors in \Cref{lem:asympvec}.
\end{proof}

These results are already surprising. In fact, if $B$ is known, the asymptotic expansions depend only on $B$, $E$, $x_i$ and $y_i$ and do not require any knowledge of other eigenvalues or other eigenvectors. Furthermore, we have

\[ \sum_{i} B(i,:)E(:,i) = 0 \qquad \sum_iB(i,:)EB(:,i) = 0 \]
so these expansions preserve the sum of the eigenvalues.
Finally, we recall the classical expression for the derivative of an eigenvalue

\begin{lemma}\label{lem:derivative}
    Let $A \in \C^{n \times n}$ and let $E \in \C^{n \times n}$ a perturbation. Let $A(t) = A + tE$. The derivative of the $i$-th eigenvalue of $A + tE$, denoted as $\lambda_i(t)$, is equal to
    \[ \frac{d\lambda_i(t)}{dt} = \frac{y_i(t)^*Ex_i(t)}{y_i(t)^*x_i(t)}\]
    where $y_i(t)$, $x_i(t)$ are the left and right eigenvectors of $\lambda_i(t)$.
\end{lemma}

In the following section, we provide non-asymptotic results. The following quantity will be of crucial importance in the next sections:
\[ \rho_i = \frac{\|E\|}{\eta_i}\]
where
\[ \eta_i = \min_{\lambda_j,\ j \neq i} |\lambda_i - \lambda_j|.\] is referred to as the spectral gap relative to the $i$-th eigenvalue.

\subsection{Non-asymptotic bounds on perturbed eigenvectors}
In this section, we provide non-asymptotic bounds for eigenvectors. In the literature, perturbed invariant subspaces (including eigenvectors) have been studied in various works \cite{Kar2014}  \cite[\S V]{Ste1990} \cite[\S 2 \S 4]{Cha2011} \cite{Mey1988}.
Consider again the matrix $D+tE$ where $t$ ranges between $0$ and $1$. Let $x_i(t)$ denote the $i$-th right eigenvector and $y_i(t)$ the $i$-th eigenvector, both normalized to have the $i$-th entry equal to $1$ (we can do this if the norm of the perturbation is small enough, see for example \cite[\S V]{Ste1990}). Let also $\eta_i = \min_{i \neq j} |\lambda_i - \lambda_j|$ be the $i$-th spectral gap.
For simplicity, we give bounds only on $x_i(t)$ ($y_i(t)$ is analogous). Recall the asymptotic expansion for eigenvectors in \Cref{lem:asympvec}
    \[ x_i(t) = e_i - tB(:,i) + t^2(D-\lambda_iI)^{\dagger}EB(:,i) + \mathcal{O}((t\|E\|)^3)\]
Define 
\begin{align}\label{eq:x_ir_i}
\begin{split}
    &x_i^{(0)} := e_i \quad
    x_i^{(1)} := -B(:,i)\quad
    x_i^{(2)} :=  (D-\lambda_iI)^{\dagger}EB(:,i)\\
    &\Delta_i(t) := \lambda_i(t) - \lambda_i\quad
    \Pi_i[z] := (I - e_ie_i^*)z\\
    &r_i^{(0)}(t) := x_i(t) - x_i^{(0)} \quad
    r_i^{(1)}(t) := x_i(t) - x_i^{(0)}-tx_i^{(1)} \quad
    r_i^{(2)}(t) := x_i(t) - x_i^{(0)}-tx_i^{(1)}-t^2x_i^{(2)}.
\end{split}
\end{align}
In other words, $x_i^{(k)}$ is the coefficient of $t^k$ in the asymptotic expansion of $x_i(t)$ and $r_i^{(k)}(t)$ is the residual relative to the asymptotic expansion of order $k$.
We define $y_i^{(\cdot)}$ and $\ell_i^{(\cdot)}(t)$ the analogous of $x_i^{(\cdot)}$ and $r_i^{(\cdot)}(t)$ for $y_i(t)$. The main results of this section are the following.

\begin{theorem}[Non-asymptotic bounds on eigenvectors]\label{thm:boundeigenvec}
Let $D$ be a diagonal matrix with eigenvalues $\{\lambda_i\}_{i=1}^n$ on the diagonal. Let $E$, $B$, $x_i(t)$ as defined before. Recall the definition 
\[ \rho_i = \frac{\|E\|}{\eta_i}.\]
Then $r_i^{(k)}(t)$ as in \Cref{eq:x_ir_i} satisfy
\begin{align*}
    \|r_i^{(0)}(t)\| &\leq \frac{t\|B(:,i)\|}{1 - \frac{|\Delta(t)| +t\|E\|}{\eta_i}} \le \frac{t\rho_i}{1-2t\rho_i}\\
     \| r_i^{(1)}(t)\| &\leq \frac{t(|\Delta(t)| + t\|E\|)\|B(:,i)\|}{\eta_i(1 -\frac{\Delta(t) + t\|E\|}{\eta_i})} \le  \frac{2(t\rho_i)^2}{1 - 2t\rho_i}\\
     \|r_i^{(2)}(t)\| &\leq  \left(|\Delta(t)|\frac{t\|B(:,i)\|}{\eta_i} + \frac{t^2(|\Delta(t)|\|E\|+t\|E\|^2)\|B(:,i)\|}{\eta_i^2}\right)\frac{1}{1-\frac{2t\|E\|}{\eta_i}} \\
     &\leq \left(t\frac{|\Delta(t)|}{\eta_i}\rho_i + 2(t\rho_i)^3\right)\frac{1}{1-2t\rho_i}.
\end{align*}
\end{theorem}
\begin{proof}
Thanks to the eigenvector definition 
\[ (D+tE)x_i(t) = \lambda_i(t)x_i(t)\]
we can write the following equations

\begin{align}\label{eq:eigeneq}
\begin{split}
    (D-\lambda_iI)r_i^{(0)}(t) &= \Delta(t)r_i^{(0)}(t) - t\Pi_i[Er_i^{(0)}(t)]- tEe_i\\
     (D-\lambda_iI)r_i^{(1)}(t) &= \Delta(t)r_i^{(0)}(t) -t\Pi_i[Er_i^{(0)}(t)]\\
     (D-\lambda_iI)r_i^{(2)}(t) &= \Delta(t)r_i^{(0)}(t)-t\Pi_i[Er_i^{(1)}(t)].
\end{split}
\end{align}
\paragraph{Linear bound} From the first equation in \Cref{eq:eigeneq} we get
\[ \|r_i^{(0)}(t)\| \le \frac{|\Delta(t)|}{\eta_i}\|r_i^{(0)}(t)\| +t\frac{\|E\|}{\eta_i}\|r_i^{(0)}(t)\|  + t\|B(:,i)\|\]
where we have used that $(D-\lambda_iI)^{\dagger}Ee_i = Be_i$ as for \Cref{lem:compB}. Hence,
\begin{equation}\label{eq:boundeigvec1}
\|r_i^{(0)}(t)\| \leq \frac{t\|B(:,i)\|}{1 - \frac{|\Delta(t)| +t\|E\|}{\eta_i}}
\end{equation}
\paragraph{Quadratic bound} From the second equation in \Cref{eq:eigeneq} and from \Cref{eq:boundeigvec1} we get
\begin{equation}\label{eq:boundeigvec2}
\|r_i^{(1)}(t)\| =  \frac{t(|\Delta(t)| +t\|E\|)\|B(:,i)\|}{\eta_i\left(1 - \frac{|\Delta(t)| +t\|E\|}{\eta_i}\right)}
\end{equation}
\paragraph{{Cubic\protect\footnote{See next section for a complete proof.}} bound} From the third equation of \Cref{eq:eigeneq}, \Cref{eq:boundeigvec1} and \Cref{eq:boundeigvec2} we get
\begin{equation}\label{eq:boundeigvec3}
    \|r_i^{(2)}(t)\| \leq \left(|\Delta(t)|\frac{t\|E\|}{\eta_i^2} + \frac{t^2(|\Delta(t)|\|E\|^2+t\|E\|^3)}{\eta_i^3}\right)\frac{1}{1-\frac{2t\|E\|}{\eta_i}}
\end{equation}
where we have used $\|B(:,i)\| \le \|E\|/\eta_i$ 
from~\Cref{eq:Bi}
and $|\Delta(t)| \le \|E\|$ by Bauer-Fike.
\end{proof}

We make a few remarks on \Cref{thm:boundeigenvec}. First, the bound
\[ \|B(:,i)\| \le \frac{\|E\|}{\eta_i} = \rho_i\]
may be loose if $E$ has a particular structure, for example if $E$ has larger entries $E(j,i)$ where $|\lambda_j - \lambda_i|$ is also large. This is because
\[ \|B(:,i)\|^2 = \sum_j\frac{E^2_{j,i}}{(\lambda_j - \lambda_i)^2}.\]
Second, the term
\[ |\Delta(t)|\frac{t\|E\|}{\eta_i^2}\]
in the last inequality of \Cref{thm:boundeigenvec} must be cubic in $\|E\|$ in order for that bound to be meaningful, as we are bounding an approximation of order $\|E\|^2$. Equivalently, we need to prove that $|\Delta(t)|$ is of order $\|E\|^2$, a result stronger than the simpler Bauer-Fike and similar to the Li and Li bound \cite{Li2005}. The proof of this result will be the central result of the next section. In \Cref{fig:unif}, \Cref{fig:sqrroot} and \Cref{fig:sqr} we plot the bounds in \Cref{thm:boundeigenvec} for some example matrices. In these figures we can see that bounding $\|B(:,i)\|$ with $\|E\|/\eta_i$ makes the bound some order of magnitudes larger. 

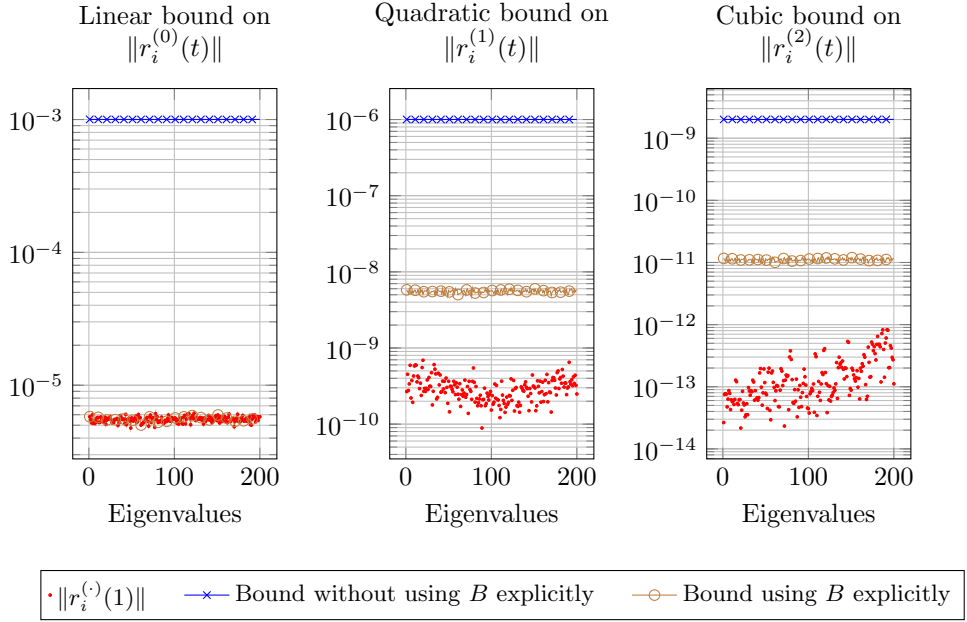
\begin{figure}\label{fig:unif}
\centering
\begin{tikzpicture}
\begin{groupplot}[
    group style={
        group size=3 by 1,
        horizontal sep=1.5cm
    },
    width = 0.33\textwidth,
    height = 0.5\textwidth,
    ymode=log,
    legend style={font=\small},
    grid=both,
    xlabel = {Eigenvalues},
]

\nextgroupplot[
align =center,
    title={Linear bound on \\ $\|r_i^{(0)}(t)\|$}
]

\addplot[mark = *, only marks,  mark size=0.5pt,  color=red] table[x expr=\thisrow{n}, y =err] {experiments/linearequi.tex};

\addplot[mark = x, color=blue, mark repeat=10] table[x expr=\thisrow{n}, y =bound] {experiments/linearequi.tex};

\addplot[mark = o, color=brown, mark repeat=10] table[x expr=\thisrow{n}, y =boundsharp] {experiments/linearequi.tex};

\nextgroupplot[
align =center,
    title={Quadratic bound on \\ $\|r_i^{(1)}(t)\|$}
]

\addplot[mark = *, only marks,  mark size=0.5pt,  color=red] table[x expr=\thisrow{n}, y =err] {experiments/quadraticequi.tex};
\addplot[mark = x, color=blue, mark repeat=10] table[x expr=\thisrow{n}, y =bound] {experiments/quadraticequi.tex};
\addplot[mark = o, color=brown,  mark repeat=10] table[x expr=\thisrow{n}, y =boundsharp] {experiments/quadraticequi.tex};

\nextgroupplot[
align =center,
    title={Cubic bound on \\ $\|r_i^{(2)}(t)\|$},
    legend columns=3,
    legend style={
        at={(-1,-0.3)},
        anchor=north,
        /tikz/every even column/.append style={column sep=12pt}
    }
]

\addplot[mark = *, only marks,  mark size=0.5pt,  color=red] table[x expr=\thisrow{n}, y =err] {experiments/cubicequi.tex};
\addlegendentry{$\|r_i^{(\cdot)}(1)\|$}
\addplot[mark = x, color=blue, mark repeat=10] table[x expr=\thisrow{n}, y =bound] {experiments/cubicequi.tex};
\addlegendentry{Bound without using $B$ explicitly}
\addplot[mark = o, color=brown,  mark repeat=10] table[x expr=\thisrow{n}, y =boundsharp] {experiments/cubicequi.tex};
\addlegendentry{Bound using $B$ explicitly}

\end{groupplot}

\end{tikzpicture}
\caption{Plot of bounds in \Cref{thm:boundeigenvec}. The red dots are the actual norm of the residuals, the blue line is the bound when we have used $\|B(:,i)\| \le \|E\|/\eta_i$ and the brown circles are the bound where $\|B(:,i)\|$ is kept as it is. $|\Delta(t)|$ in is bounded by \Cref{thm:boundeige2}. The matrix considered is a diagonal matrix with eigenvalues equispaced from $1$ to $200$ and $B$ is a random off-diagonal matrix such that $\|E\| = 10^{-3}$. Note that $\rho_i$ is constant and equal to $10^{-3}$. The blue line in the linear bound is $ \approx \|E\|/\eta_i$, which is the classical bound presented for perturbed eigenvectors. 
}
\end{figure}

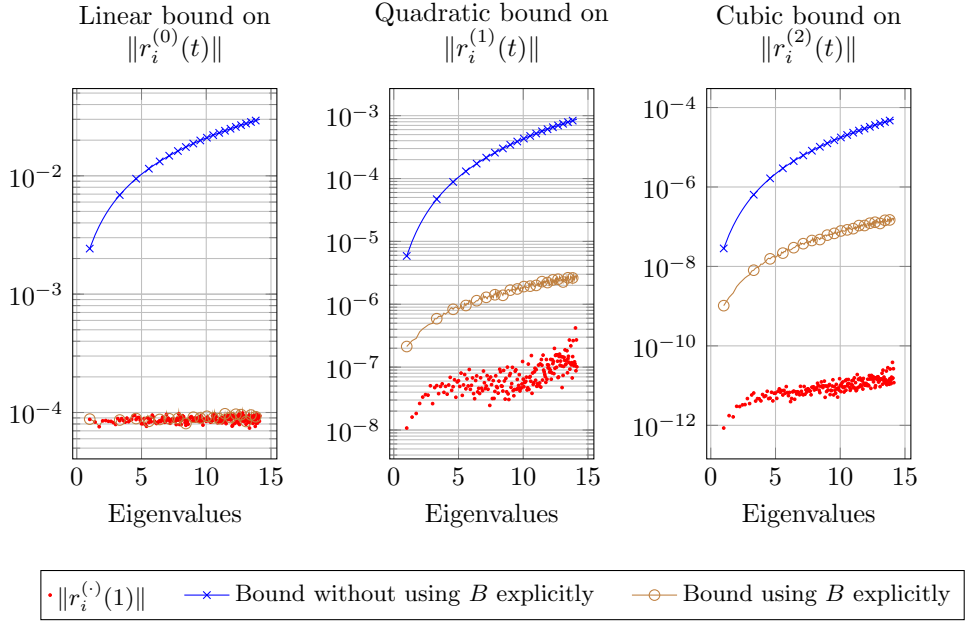
\begin{figure}\label{fig:sqrroot}
\centering
\begin{tikzpicture}
\begin{groupplot}[
    group style={
        group size=3 by 1,
        horizontal sep=1.5cm
    },
    width = 0.33\textwidth,
    height = 0.5\textwidth,
    ymode=log,
    legend style={font=\small},
    grid=both,
    xlabel = {Eigenvalues},
]

\nextgroupplot[
align =center,
    title={Linear bound on \\ $\|r_i^{(0)}(t)\|$}
]

\addplot[mark = *, only marks,  mark size=0.5pt,  color=red] table[x expr=sqrt(\thisrow{n}), y =err] {experiments/linearsqrroot.tex};

\addplot[mark = x, color=blue, mark repeat=10] table[x expr=sqrt(\thisrow{n}), y =bound] {experiments/linearsqrroot.tex};

\addplot[mark = o, color=brown,mark repeat=10] table[x expr=sqrt(\thisrow{n}), y =boundsharp] {experiments/linearsqrroot.tex};

\nextgroupplot[
align =center,
    title={Quadratic bound on \\ $\|r_i^{(1)}(t)\|$}
]

\addplot[mark = *, only marks,  mark size=0.5pt,  color=red] table[x expr=sqrt(\thisrow{n}), y =err] {experiments/quadraticsqrroot.tex};
\addplot[mark = x, color=blue, mark repeat=10] table[x expr=sqrt(\thisrow{n}), y =bound] {experiments/quadraticsqrroot.tex};
\addplot[mark = o, color=brown,  mark repeat=10] table[x expr=sqrt(\thisrow{n}), y =boundsharp] {experiments/quadraticsqrroot.tex};

\nextgroupplot[
align =center,
    title={Cubic bound on \\ $\|r_i^{(2)}(t)\|$},
    legend columns=3,
    legend style={
        at={(-1,-0.3)},
        anchor=north,
        /tikz/every even column/.append style={column sep=12pt}
    }
]

\addplot[mark = *, only marks,  mark size=0.5pt,  color=red] table[x expr=sqrt(\thisrow{n}), y =err] {experiments/cubicsqrroot.tex};
\addlegendentry{$\|r_i^{(\cdot)}(1)\|$}
\addplot[mark = x, color=blue, mark repeat=10] table[x expr=sqrt(\thisrow{n}), y =bound] {experiments/cubicsqrroot.tex};
\addlegendentry{Bound without using $B$ explicitly}
\addplot[mark = o, color=brown, mark repeat=10] table[x expr=sqrt(\thisrow{n}), y =boundsharp] {experiments/cubicsqrroot.tex};
\addlegendentry{Bound using $B$ explicitly}

\end{groupplot}

\end{tikzpicture}
\caption{Plot of bounds in \Cref{thm:boundeigenvec}. The red dots are the actual norm of the residual, the blue line is the bound when we have used $\|B(:,i)\| \le \|E\|/\eta_i$ and the brown circles are the bound where $\|B(:,i)\|$ is kept as it is. $|\Delta(t)|$ in is bounded by \Cref{thm:boundeige2}. The matrix considered is a diagonal matrix with eigenvalues equal to the square root of the  equispaced points from $1$ to $200$ and $B$ is a random off-diagonal matrix such that $\|E\| = 10^{-3}$. Note that, for the largest eigenvalues, $\rho_i \approx 10^{-1}$. The blue line in the linear bound is $ \approx \|E\|/\eta_i$, which is the classical bound presented for perturbed eigenvectors. }
\end{figure}

\begin{figure}\label{fig:sqr}
\centering
\begin{tikzpicture}
\begin{groupplot}[
    group style={
        group size=3 by 1,
        horizontal sep=1.5cm
    },
    width = 0.33\textwidth,
    height = 0.5\textwidth,
    ymode=log,
    legend style={font=\small},
    grid=both,
    xlabel = {Eigenvalues},
]

\nextgroupplot[
align =center,
    title={Linear bound on \\ $\|r_i^{(0)}(t)\|$}
]

\addplot[mark = *, only marks,  mark size=0.5pt,  color=red] table[x expr=(\thisrow{n})^2, y =err] {experiments/linearpow2.tex};

\addplot[mark = x, color=blue, mark repeat=10] table[x expr=(\thisrow{n})^2, y =bound] {experiments/linearpow2.tex};

\addplot[mark = o, color=brown, mark repeat=10] table[x expr=(\thisrow{n})^2, y =boundsharp] {experiments/linearpow2.tex};

\nextgroupplot[
align =center,
    title={Quadratic bound on \\ $\|r_i^{(1)}(t)\|$}
]

\addplot[mark = *, only marks,  mark size=0.5pt, color=red] table[x expr=(\thisrow{n})^2, y =err] {experiments/quadraticpow2.tex};
\addplot[mark = x, color=blue, mark repeat=10] table[x expr=(\thisrow{n})^2, y =bound] {experiments/quadraticpow2.tex};
\addplot[mark = o, color=brown, mark repeat=10] table[x expr=(\thisrow{n})^2, y =boundsharp] {experiments/quadraticpow2.tex};

\nextgroupplot[
align =center,
    title={Cubic bound on \\ $\|r_i^{(2)}(t)\|$},
    legend columns=3,
    legend style={
        at={(-1,-0.3)},
        anchor=north,
        /tikz/every even column/.append style={column sep=12pt}
    }
]

\addplot[mark = *, only marks,  mark size=0.5pt,  color=red] table[x expr=(\thisrow{n})^2, y =err] {experiments/cubicpow2.tex};
\addlegendentry{$\|r_i^{(\cdot)}(1)\|$}
\addplot[mark = x, color=blue, mark repeat=10] table[x expr=(\thisrow{n})^2, y =bound] {experiments/cubicpow2.tex};
\addlegendentry{Bound without using $B$ explicitly}
\addplot[mark = o, color=brown, mark repeat=10] table[x expr = (\thisrow{n})^2, y =boundsharp] {experiments/cubicpow2.tex};
\addlegendentry{Bound using $B$ explicitly}

\end{groupplot}

\end{tikzpicture}
\caption{Plot of bounds in \Cref{thm:boundeigenvec}. The red dots are the actual norm of the residual, the blue line is the bound when we have used $\|B(:,i)\| \le \|E\|/\eta_i$ and the brown circles are the bound where $\|B(:,i)\|$ is kept as it is. $|\Delta(t)|$ in is bounded by \Cref{thm:boundeige2}. The matrix considered is a diagonal matrix with eigenvalues equal to the square of the  equispaced points from $1$ to $200$ and $B$ is a random off-diagonal matrix such that $\|E\| = 0.3$. Note that, for the largest eigenvalues, $\rho_i \approx 10^{-3}$. The blue line in the linear bound is $ \approx \|E\|/\eta_i$, which is the classical bound presented for perturbed eigenvectors. }
\end{figure}
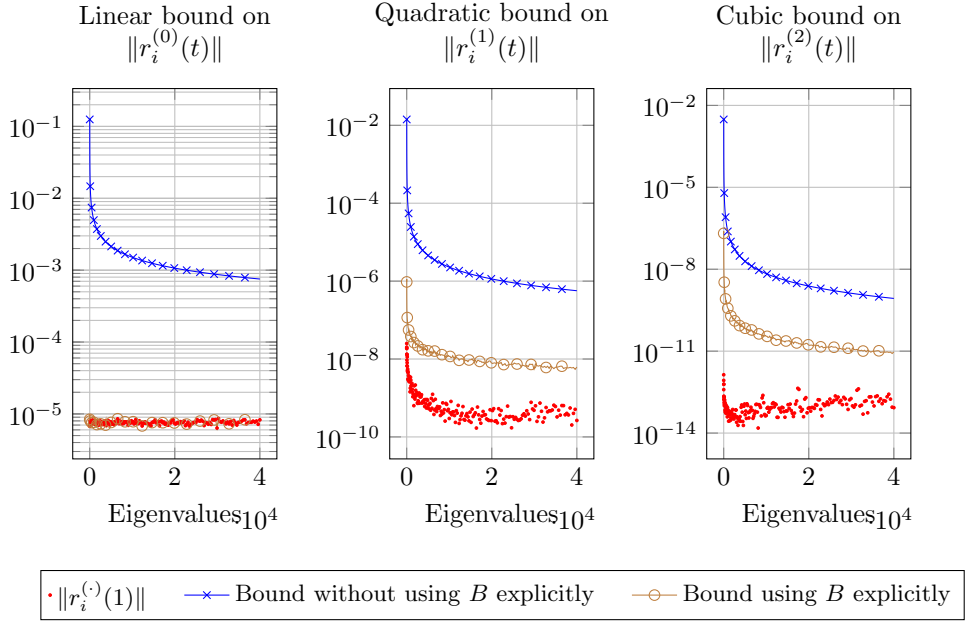

\subsection{Non-asymptotic bounds on perturbed eigenvalues using eigenvector asymptotic expansion}

In this section, we provide non-asymptotic bounds for eigenvalues using \Cref{thm:boundeigenvec}. The main theorem is the following.

\begin{theorem}[Non-asymptotic bounds on eigenvectors]\label{thm:boundeige2}
Let $D$ a diagonal matrix with eigenvalues $\{\lambda_i\}_{i=1}^n$ on the diagonal. Let $E$, $B$, $x_i(t)$ be as defined before. Recall the definition 
\[ \rho_i = \frac{\|E\|}{\eta_i}.\]
We have the following;
\begin{align*}
    |\lambda_i(t) - \lambda_i| \le \frac{t^2\|E\|\frac{\|B(i,:)\|+\|B(:,i)\|}{2} + \frac{t^3\|E\|\|B(i,:)\|\|B(:,i)\|}{3\left(1 - \frac{2\|E\|}{\eta_i}\right)}}{1 - \frac{2\|E\|}{\eta_i} - \frac{\|B(i,:)\|\|B(:,i)\| }{\left(1 - \frac{2\|E\|}{\eta_i}\right)}} \le \frac{t^2\|E\|(\rho_i + \frac{t}{3} \frac{\rho_i^2}{1-\rho_i})}{1-2\rho_i-\frac{\rho_i^2}{1 - 2 \rho_i}}.
\end{align*}
\end{theorem}

\begin{proof}

The proof relies on the formula for the derivative of eigenvalues \Cref{lem:derivative} $\frac{d\lambda_i(t)}{dt} = \frac{y_i(t)^*Ex_i(t)}{y_i(t)^*x_i(t)}.$
Consider first the decompositions
\begin{align*}
    x_i(t) &= x_i^{(0)} + r_i^{(0)}(t)\\
    y_i(t) &=  y_i^{(0)}+\ell^{(0)}_i(t).
\end{align*}
We can rewrite the derivative as
\[ \frac{d\lambda_i(t)}{dt} = \frac{y_i(t)^*Ex_i(t)}{y_i(t)^*x_i(t)} = \frac{(y_i^{(0)}+\ell^{(0)}_i(t))^*E(x_i^{(0)} + r_i^{(0)}(t))}{(y_i^{(0)}+\ell^{(0)}_i(t))^*(x_i^{(0)} + r_i^{(0)}(t))} =\] 
\[ = \frac{\ell^{(0)}_i(t)^*Ex_i^{(0)} + y_i^{(0)}Er_i^{(0)}(t) + \ell_i^{(0)}(t)^*Er_i^{(0)}(t)}{1 + \ell_i^{(0)}(t)^*r_i^{(0)}(t)}\]
where we have used that $x_i^{(0)}, y_i^{(0)} \perp r_i^{(0)}(t), \ell_i^{(0)}(t)$.
From \Cref{thm:boundeigenvec} we have
\begin{align*}
    \|r_i^{(0)}(t)\| &\le \frac{t\|B(:,i)\|}{1 - \frac{2\|E\|}{\eta_i}}\\
    \|\ell_i^{(0)}(t)\| &\le \frac{t\|B(i,:)\|}{1 - \frac{2\|E\|}{\eta_i}}
\end{align*}
which implies
\[ \left|\frac{d\lambda_i(t)}{dt}\right| \le \frac{t\|E\|(\|B(i,:)\| + \|B(:,i)\|) + \frac{t^2\|E\|\|B(i,:)\|\|B(:,i)\|}{\left(1 - \frac{2\|E\|}{\eta_i}\right)}}{1 - \frac{2\|E\|}{\eta_i} - \frac{\|B(i,:)\|\|B(:,i)\|}{\left(1 - \frac{2\|E\|}{\eta_i}\right)}}.\]
By integration, we get the bound.
\end{proof}

This theorem implies that the perturbation on the eigenvalues is quadratic in $\|E\|$ and in particular justifies that last inequality in \Cref{thm:boundeigenvec} is $\mathcal{O}(\|E\|^3)$.

\subsection{Further improvements for eigenvalue bounds}

As we have at our disposal non-asymptotic bounds for eigenvectors up to a second order correction, we should in principle be able to plug them into the expression of the derivative of eigenvalues and get non-asymptotic bound for eigenvalues of high order. Unfortunately, the number of terms in the denominator of \Cref{lem:derivative} grows quadratically with the order of the bound on the eigenvectors and the calculation is cumbersome. Fortunately, If we make a further assumption on $\rho_i$, for example
\[ \rho_i \le \frac{1}{4}\]
we may in principle prove two other theorems.

\begin{theorem}
Let $D \in \C^{n\times n}$ be a diagonal matrix, and $E \in \C^{n\times n}$ with $diag(E) = 0$ and the matrix $B \in \C^{n \times n}$ such that $E = AB-BA$ and $diag(B) = 0$. Let $\lambda_i$ the $i$-th eigenvalue of $D$ and $\lambda_i(1)$ the $i$-th eigenvalue of $D+E$. Then, if
\[ \rho_i = \frac{\|E\|}{\eta_i} \leq \frac{1}{4}\]
it holds that
\[ |\lambda_i(1) - \lambda_i -B(i,:)E(:,i)| \leq C_1 \frac{\|E\|^3}{\eta_i^2} = C_1\|E\|\rho_i^2\]
where $C_1$ is a constant independent from all the quantities involved.
\end{theorem}

\begin{theorem}
    Given a diagonal matrix $D \in \C^{n\times n}$, a generic matrix $E \in \C^{n\times n}$ with $diag(E) = 0$ and the matrix $B \in \C^{n \times n}$ such that $E = AB-BA$ and $diag(B) = 0$.  Let $\lambda_i$ be the $i$-th eigenvalue of $D$ and $\lambda_i(1)$ the $i$-th eigenvector of $D+E$. Then, if
    \[ \frac{\|E\|}{\eta_i} \leq \frac{1}{4}\]
    It holds that
    \[ |\lambda_i(1) -\lambda_i - B(i,:)E(:,i) +B(i,:)EB(:,i)| \leq C_2\frac{\|E\|^4}{\eta_i^3}\]
    where $C_2$ is a constant independent from all the quantities involved.
\end{theorem}

\section{Generic perturbation $E$}\label{sec:gener}

In this section, we generalize previous results to the case where the perturbation $E$ does not lie in the tangent space. As before, we will reduce the problem to the diagonal case and consider the diagonal matrix $D$. Define $B$ as 
\[ DB-BD = E -diag(E)\]
where $diag(E)$ is the diagonal matrix with diagonal equal to the diagonal of $E$.
Note that 
\[ E = (E -diag(E)) +diag(E)\]
is an additive decomposition in
\[ \mathcal{T}_D{\mathcal{M}} \times \mathcal{T}_D{\mathcal{M}}^{\perp}\]
where $\mathcal{M}$ is the isospectral manifold and $\mathcal{T}_D{\mathcal{M}}$ is the tangent space in $D$. In fact, if we use the natural metric,
\[ \langle (E -diag(E),diag(E)\rangle  = \operatorname{tr}((E -diag(E))^*diag(E)) = 0\]
as we have the product with a zero diagonal matrix and a diagonal matrix has a zero diagonal, and by the characterization \Cref{lem:zerodiag}, the first matrix is in the tangent space. 
Consider again \Cref{eq:eigeneq}. If $E$ has non zero diagonal, the new equations are
\begin{align*}
    \begin{split}
    (D-\lambda_iI)r_i^{(0)}(t) &= \Delta(t)r_i^{(0)}(t) + tEr_i^{(0)}(t) + t(E -E(i,i)I)e_i\\
     (D-\lambda_iI)r_i^{(1)}(t) &= \Delta(t)r_i^{(0)}(t) -t\Pi_i[Er_i^{(0)}(t)]\\
     (D-\lambda_iI)r_i^{(2)}(t) &= (\Delta(t)-tE(i,i))r_i^{(0)}(t)-t\Pi_i[(E- E(i,i)I)r_i^{(1)}(t)].
    \end{split}
\end{align*}

Using this equations, we can generalize \Cref{thm:boundeigenvec} to this case and also the bound on eigenvalues, keeping in mind that now the asymptotic expansions of eigenvectors are
\[  x_i(t) \approx e_i -tB(:,i) +t^2(D-\lambda_iI)^{-1}EB(:,i) + t^2E(i,i)(D-\lambda_iI)^{-1}B(:,i)\]
while the expansions for eigenvalues are 

\[ \lambda_i(t) \approx \lambda_i + tE(i,i) + t^2B(i,:)E(:,i).\]

\section{$2\times2$ block-diagonal matrices}\label{sec:2x2}

Assume now that $D$ is diagonal, but divided into two blocks

\[ D = \begin{bmatrix}
    D_1 & 0 \\
    0 & D_2
\end{bmatrix}.\]
Inside each block, eigenvalues can be multiple or close to each other. In particular, we assume for the moment that $D_1 = \lambda_1I$. Assume that $\lambda_1$ has a distance $gap$ from the spectrum of $D_2$, that is
\[ gap = \min_{ \lambda \in \Lambda(D_2)} |\lambda_1 - \lambda|.\]
Consider the perturbation
\[ E = \begin{bmatrix}
    0 & E_1 \\
    E_2 & 0
\end{bmatrix}\]
where the dimensions of $E_1$ and $E_2$ match the partitioning of $D$. In this case, even if we have multiple eigenvalues, the matrix $B$ is well defined. In particular $B$ is (can be chosen) equal to 
\[ B = \begin{bmatrix}
    0 & B_1 \\
    B_2 & 0
\end{bmatrix}\]
and
\[
    E_1 = D_1B_1 - B_1D_2,\qquad  E_2 = D_2B_2 - B_2D_1\]
Note that, unlike the general case, the entries of $B$ only involve differences between eigenvalues that belong to different blocks (in this case, only differences between $\lambda_1$ and one eigenvalue of $D_2$). 
In this setting, we want to give a first order approximation to the eigenspace relative to the eigenvectors of the first block, similar to what we have done before. Note that, as the first eigenvalue is multiple, we are not able to give any result on the single eigenvector, but we can only analyse the eigenspace on the whole.
Similarly to the single eigenvector, we want to prove that the eigenspace is well approximated by
\[ span\left(\begin{bmatrix}
    I\\
    -tB_2
\end{bmatrix}\right)\]
which is the block analogous of the first order expansion $e_i -tB(:,i)$ that we had for $x_i(t)$.

Let $x_i(t)$ be the normalized right eigenvector relative to $\lambda_i(t)$. Consider the decomposition of $x_i(t)$ as
\begin{equation}\label{eq:xidecomp}
x_i(t) = \begin{bmatrix}
    I\\
    -tB_2
\end{bmatrix}x_i^{(top)}(t) + \begin{bmatrix}
    0\\
    r_i(t)
\end{bmatrix}
\end{equation}
where $x_i^{(top)}(t)$ is the vector of the same length as the width of $D_1$ whose entries are the first entries of $x_i(t)$. We want to show that $r_i(t)$ is bounded by a quadratic term in $\|E\|$.
We consider again the eigenvector equation
\[ (D+tE)x_i(t) = \lambda_i(t)x_i(t)\]
\[ (D+tE)(\begin{bmatrix}
    I\\
    -tB_2
\end{bmatrix}x_i^{(top)}(t) + \begin{bmatrix}
    0\\
    r_i(t)
\end{bmatrix}) = \lambda_i(t)(\begin{bmatrix}
    I\\
    -tB_2
\end{bmatrix}x_i^{(top)}(t) + \begin{bmatrix}
    0\\
    r_i(t)
\end{bmatrix}).\]
If we look at the last rows, we get
\begin{equation}\label{eq:lastrows}
(D_2 - \lambda_i(t)I)r_i(t) = tB_2(D_1 - \lambda_i(t)I)x_i^{(top)}(t)
\end{equation}
where we have used the definition of $E_2$ to rewrite $D_2B_2$. This implies
\[ \| r_i(t)\| \le \frac{t^2\|B_2\|\|E\|}{gap -\|E\|}\]
where we have used Bauer-Fike to bound $\|D_1-\lambda_i(t)I\|$ and $\|D_2-\lambda_i(t)I\|$.
We have proven the following theorem.

\begin{theorem}\label{thm:blkdiag}
    Consider a diagonal matrix $D \in \C^{n \times n}$ of the form
    \[ D = \begin{bmatrix}
    D_1 & 0 \\
    0 & D_2
\end{bmatrix}\]
where $D_1 = \lambda_1I$ and let
\[ gap = \min_{ \lambda \in \Lambda(D_2)} |\lambda_1 - \lambda|>0\]
be the spectral gap.
Let $E$ and $B$ have the following partition
\[ E = \begin{bmatrix}
    0 & E_1 \\
    E_2 & 0
\end{bmatrix}, \quad B = \begin{bmatrix}
    0 & B_1 \\
    B_2 & 0
\end{bmatrix}\]
where the block sizes match those of $A$, and $B_1$ satisfies
\[
    E_1 = D_1B_1 - B_1D_2, \quad  E_2 = D_2B_2-B_2D_1.\]
Let $\mathcal{X}$ be
\[ \mathcal{X} = span\left(\begin{bmatrix}
    I\\
    -B_2
\end{bmatrix}\right)\]
Then, for every normalized right eigenvector $x_i(t)$ of  $D + tE$ relative to a perturbed eigenvalue of the upper-left block of $D$, we have 
\[ \|\Pi_{\mathcal{X}^{\perp}}x_i(t)\|\le  \frac{\|B_2\|\|E\|}{gap -\|E\|}\]
where $\Pi_{\mathcal{X}^{\perp}}$ is the projection on $\mathcal{X}^\perp$.

\end{theorem}

\subsection{$D_1$ not a multiple of the identity}\label{sub:d1notidentity}

Here, we provide a similar bound assuming that $D_1$ is diagonal with eigenvalues in the closed ball 
\[ \bar{B}(\lambda_1; \Delta/2)\]
that is, we allow $D_1$ not to be a multiple of the identity, but we assume that $\Delta$ is comparable to $\|E\|$ or smaller.
From \Cref{eq:lastrows}
\[ (D_2 - \lambda_i(t))r_i(t) = tB_2(D_1 - \lambda_i(t))x_i^{(top)}(t)\]
where $r_i(t)$ and $x_i^{(top)}(t)$ are defined as in \Cref{eq:xidecomp},
we obtain
\[ \| r_i(t)\| \le \frac{t^2\|B_2\|(\|E\| + \Delta)}{gap -\|E\|}\]
where we have used Bauer-Fike. 
Note that, in this setting, even if $\|E_1\|$ in the numerator is improved with a quadratic term (i.e. we use a stronger result than Bauer-Fike), the term
\[ t^2\|B_2\|\Delta\]
cannot be improved. 
We provide stronger results if the matrices are Hermitian in \Cref{sub:2x2herm}.

\section{Hermitian matrices}\label{sec:hermi}

\subsection{A first simple bound for Hermitian matrices}
In this section, we show how the situation simplify if the matrix $A$ is Hermitian. The following simple result holds without assuming $A$ is diagonal or without assuming separation in the spectrum.

\begin{theorem}
Let $A \in \C^{n \times n}$ be Hermitian and $B \in \C^{n \times n}$ be skew-Hermitian.
Define 
\[ E := AB-BA.\]
As $B$ is skew-Hermitian, $E$ is Hermitian.
Then  
\[
|\lambda_i(A+E) - \lambda_i(A)|\leq \|B\|_2\|AB-BA\|_2 = \|B\|_2\|E\|_2.
\]  
\end{theorem}

\begin{proof}

Suppose 
\begin{equation}  \label{eq:lamab}
(A+t(AB-BA))x_i(t) = \lambda_i(t)x_i(t)  
\end{equation}
where $x_i(t)$ is the normalized eigenvector associated with $\lambda_i(t)$ for all $t$ (the nonuniqueness of the eigenvector when $\lambda_i(t)$ is multiple appears to be an issue, but in fact it is not~\cite{Nak2012}).
Then, by standard eigenvalue perturbation theory, we have 
\[
\frac{d\lambda_i(t)}{dt} = x_i(t)^*(AB-BA)x_i(t).
\]
Now, using~\Cref{eq:lamab}, we have 
$Ax_i(t)  = (\lambda_i(t)I -t(AB-BA))x_i(t)$, so 
\begin{align*}
\frac{d\lambda_i(t)}{dt} 
&= x_i(t)^*(AB-BA)x_i(t)\\
&= ((\lambda_i(t)I -t(AB-BA))x_i(t))^*Bx_i(t)
- x_i(t)^*B (\lambda_i(t)I -t(AB-BA))x_i(t)\\
&= x_i(t)^*B((\lambda_i(t)I +t(AB-BA))x_i(t))
- x_i(t)^*B (\lambda_i(t)I -t(AB-BA))x_i(t)\\
&= tx_i^*(t)B(AB-BA)x_i(t) + tx_i^*(t)(AB-BA)Bx_i(t).
\end{align*}
We thus have 
$|\frac{d\lambda_i(t)}{dt} |\leq 2|t| \|B\|_2\|AB-BA\|_2.$
Therefore, 
\[
|\lambda_i(1)-\lambda_i(0)|\leq \int_{0}^12|t| \|B\|_2\|AB-BA\|_2dt
= \|B\|_2\|AB-BA\|_2 =\|B\|_2\|E\|_2 .
\]
\end{proof}

This bound is elegant and does not require diagonalizing the matrix or any result on eigenvectors. Additionally, the bound is quadratic, as
\[ \|B\|_2 \le \frac{\|E\|_F}{\eta}\]
where $\eta$ is defined as the global spectral gap
\[ \eta = \min_{\lambda_i,\ \lambda_j \in \Lambda(A), \lambda_i \neq \lambda_j} |\lambda_i - \lambda_j|\]
and $B$ has minimal norm among all the possible choices of $B$ (recall that $B$ has $n$ degrees of freedom).
On the other hand, this bound has a major limitation: assuming that $A$ is diagonal, the entries $B(i,j)$ scale with $\frac{1}{\lambda_i - \lambda_j}$. This implies that if we have a cluster of eigenvalues, this bound may be loose, even if $\lambda_i$ is well isolated from the rest. Not only that, if $A$ has multiple eigenvalues, we are implicitly assuming that $E$ has zero entries not only on the diagonal, but also in other principal minors (i.e. $E$ has to satisfy additional constraints). 

Before proceeding, it is worth noticing that if the matrix $A$ is Hermitian, the tangent space of the isospectral manifold always includes the skew-Hermitian matrices. This is because if $E$ is skew-Hermitian, then
\[ x^*Ex = 0 \qquad \forall x \in \C^n\]
in particular, it holds if $x$ is an eigenvector of $A$. The claim follows from \Cref{thm:chartang}.

\subsection{$2 \times 2$ block-diagonal Hermitian matrices}\label{sub:2x2herm}

In this section, we continue the argument in \Cref{sec:2x2}, and provide stronger results under the assumption of Hermiticity. We consider the Hermitian diagonal matrix
\[ D = \begin{bmatrix}
    D_1 & 0 \\
    0 & D_2
\end{bmatrix}\]
where $D_1 = \lambda_1I$.
Furthermore, the matrix $E$
\[ E = \begin{bmatrix}
    0 & E_1 \\
    E_2 & 0
\end{bmatrix}\]
is Hermitian, hence
\[ E_2 = E_1^* \quad \text{and} \quad B_2 =-B_1^*.\]
Consider the derivative \Cref{lem:derivative}
\[ \frac{d\lambda_i(t)}{dt} = x_i^*(t)Ex_i(t).\]
Using the decomposition in \Cref{eq:xidecomp}, this expression is equal to

\[ x_i^*(t)Ex_i(t) = (x_i^{(top)}(t)^*\begin{bmatrix}
    tB_1E_2 & E_1
\end{bmatrix}  + \begin{bmatrix}
    l_i(t)^*E_2 & 0
\end{bmatrix})( \begin{bmatrix}
    I\\
    -tB_2
\end{bmatrix}x_i^{(top)}(t) + \begin{bmatrix}
    0\\
    r_i(t)
\end{bmatrix}) =\]
\[ 
= tx_i^{(top)}(t)^*(
    B_1E_2 - E_1B_2)x_i^{(top)}(t) + x_i^{(top)}(t)^*E_1r_i(t) + l_i(t)^*E_2x_i^{(top)}(t).
\]
Using the result from \Cref{thm:blkdiag} (and noticing that, in this case, $\|E\| = \|E_1\|$) we have 
\[|x_i(t)^*Ex_i(t)| \le t \|B_1E_1^* + E_1B_1^*\| + \frac{2t^2\|E_1\|^2\|B_1\|}{gap - \|E_1\|}\]
Hence
\[ |\lambda_i(1) - \lambda_i| \leq \frac{1}{2}\|B_1E_1^* + E_1B_1^*\| + \frac{2}{3}\frac{\|E_1\|^2\|B_1\|}{gap-\|E_1\|}.\]
We have proven the following result:
\begin{theorem}\label{thm:blkeigen}
Given a Hermitian matrix $D$ of the form
    \[ D = \begin{bmatrix}
    D_1 & 0 \\
    0 & D_2
\end{bmatrix}\]
where $D_1 = \lambda_1I$.
Let $E$ be a Hermitian perturbation of the form
\[ E = \begin{bmatrix}
    0 & E_1 \\
    E_1^* & 0
\end{bmatrix}\]
and $B$ equal to
\[ B = \begin{bmatrix}
    0 & B_1 \\
    -B_1^* & 0
\end{bmatrix}\]
where $B_1$ satisfies
\[
    E_1 = D_1B_1 - B_1D_2. 
\]
Let $\lambda_i(1)$ the $i$-th eigenvalue of $D + E$ and let $\lambda_i$ the $i$-th eigenvalue of $D$, and assume that $\lambda_i$ is an eigenvalue of $D_1$, then
\[ |\lambda_i(1) - \lambda_i| \leq \frac{1}{2}\|B_1E_1^* + E_1B_1^*\| + \frac{2}{3}\frac{\|E_1\|^2\|B_1\|}{gap-\|E_1\|}.\]
\end{theorem}
Note that the term 
\[\frac{1}{2}\|B_1E_1^* + E_1B_1^*\| \]
hides inside $B_1$ information about the entire $D_2$ and not only the gap. This guarantees a much more precise bound when some eigenvalues of $D_2$ are far from $\lambda_1$ but the $gap$ is small.

\subsection{A further improvement}

Consider again \Cref{eq:lastrows}
\[ (D_2 - \lambda_i(t))r_i(t) = tB_2(D_1 - \lambda_i(t))x_i^{(top)}(t).\]
Instead of using Bauer-Fike, we can now use \Cref{thm:blkeigen} to bound the perturbation of eigenvalues. This yields

\[ \| r_i(t)\| \le \frac{1}{gap -\|E_1\|}\left(\frac{t^2}{2}\|B_1\|\|B_1E_1^*+E_1B_1^*\| + \frac{2}{3}t^3\frac{\|B_1\|^2\|E_1\|^2}{gap - \|E_1\|}\right).\]

\begin{theorem}\label{thm:blkvecherm}
Under the same hypothesis of \Cref{thm:blkeigen} and using the same notation as in \Cref{eq:xidecomp}, we have
\[ \| r_i(t)\| \le \frac{1}{gap -\|E_1\|}\left(\frac{t^2}{2}\|B_1\|\|B_1E_1^*+E_1B_1^*\| + \frac{2}{3}t^3\frac{\|B_1\|^2\|E_1\|^2}{gap - \|E_1\|}\right).\]
\end{theorem}
In other words,
\[ \|r_i(t)\| \le C\left(\frac{\|E\|}{gap}\right)^3\]
if $\frac{\|E\|}{\eta_i}$ is sufficiently smaller than $1$ and $C$ is a constant.
This tells us that the perturbation of the eigenspace is not quadratic, but cubic. This behaviour was noted also in \cite{Li2012} in the context of eigenvalue perturbation.

Note that, according to \Cref{sub:d1notidentity}, this improvement does not generalize if $D_1$ is not a multiple of the identity but rather has a cluster of eigenvalues of width comparable to $\|E\|$.

\begin{figure}
\centering
\begin{tikzpicture}

\begin{semilogyaxis}[
    width=0.5\textwidth,
    height=0.5\textwidth,
    grid=both,
    legend pos=outer north east,
    xlabel = $i$
]

\addplot[mark = *, only marks,  mark size=0.5pt,  color=red] table[x=n, y =err] {experiments/blockequiherm.tex};
\addlegendentry{$\|r_i(1)\|$}

\addplot[mark = x, color=blue, mark repeat=5] table[x=n,, y =bound] {experiments/blockequiherm.tex};
\addlegendentry{Bound in \Cref{thm:blkvecherm}}

\end{semilogyaxis}
\end{tikzpicture}
\caption{Bound in \Cref{thm:blkvecherm} where $D \in \R^{200,200}$, $\lambda_1 = 1$ and $D_2$ is diagonal with eigenvalues equispaced from $2$ to $101$. The matrix $B_1$ is random and scaled such that $\|E\| = 10^{-2}$. Note that this bound is much tighter than $(\|E\|/gap)^3 = 10^{-6}$.}
\end{figure}
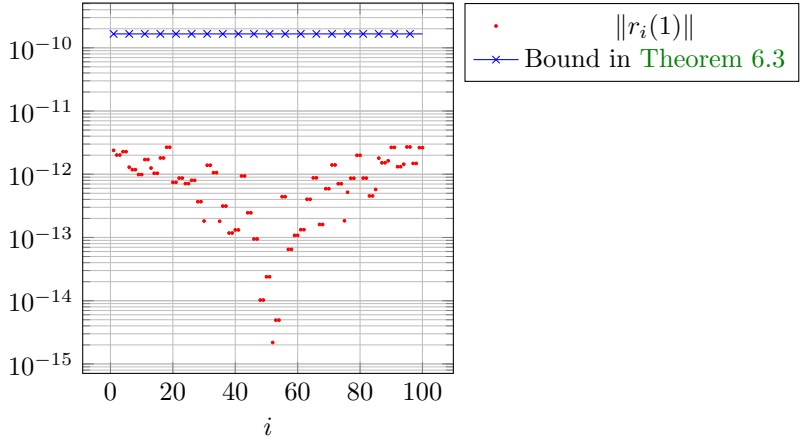

\section{SVD}\label{sec:svd}
In this section, we want to apply the same idea that we have employed for eigenvalues in the setting of singular values. Let $A \in \mathbb{C}^{m \times n}$ and let $A = U\Sigma V^*$ its SVD. The manifold of matrices with the same singular values can be parametrized as
\[ \mathbb{U}(m) \times \mathbb{U}(n) \longrightarrow \C^{m \times n}\]
\[ [P,Q] \longrightarrow P\Sigma Q^*.\]

To find the tangent space, we proceed as before: let $P_1$ and $P_2$ be two small perturbations. First of all, we have to make sure that we are moving along the tangent spaces of $\mathbb{U}(m)$ and $\mathbb{U}(n)$ respectively, which means that $U + P_1$ and $V + P_2$ are orthonormal to the first order, which implies
\[ (U+P_1)^*(U+P_1) = I +O(\|P_1\|^2) \Rightarrow (U^*P_1)^* = -(U^*P_1)\]
\[ (V+P_2)^*(V+P_2) = I +O(\|P_2\|^2)\Rightarrow (V^*P_2)^* = -(V^*P_2).\]
Now, the perturbed SVD is of the form
\[ (U+P_1)\Sigma(V+P_2)^* = A + (P_1U^*)A + A(VP_2^*).\]
As $P_1$ and $P_2$ are generic matrices, the tangent space is described by
\[ SkewHerm(m) \times SkewHerm(n) \longrightarrow \C^{m \times n}\]
\[ [S_1,S_2] \longrightarrow S_1A + AS_2.\]
Again, this matrix satisfies
\[ \mathscr{Re}(u_i^*(S_1A+AS_2)v_i) = 0,\]
which is the first order perturbation of the $i$-th singular value. The number of constraints matches the degrees of freedom of the singular vectors (associated with nonzero singular values), as the right and left singular vectors may be scaled by the same unitary diagonal matrix. 

Let $A \in \mathbb{C}^{m \times n}$ a complex matrix with $m \geq n$ and let $S_1 \in \mathbb{C}^{m \times m}$, $S_2 \in \mathbb{C}^{n \times n}$ be skew-Hermitian matrices. We are interested in studying the singular values of the matrix
\[ A + S_1A + AS_2.\]

\begin{remark}
Suppose that $A$ is diagonal (that is only $A(i,i) \neq 0$ for $i = 1, \dots, n$) and  containing the singular values on the diagonal. In this case, $E$ must have zero diagonal and $S_1$ and $S_2$ can be expressed in terms of $E$ in the following way:
\[S_1(i,j) = \frac{\sigma_j\, E(i,j) + \sigma_i\, E(j,i)}{\sigma_j^2 - \sigma_i^2}, \quad
S_2(i,j) = -\frac{\sigma_j\, E(j,i) + \sigma_i\, E(i,j)}{\sigma_j^2 - \sigma_i^2},
\quad i \neq j, i \le n, j \le n.\]
while
\[ 
S_1(i,j)=\frac{E(i,j)}{\sigma_j},\qquad i>n,\; j\le n.\
\]
The lower-right block of $S_1$ is free.
\end{remark}
Let
\[ A(t) = A + t(S_1A + AS_2)\]
and let $\sigma_i(t), u_i(t), v_i(t)$ be the $i$-th singular value/left singular vector/right singular vector of $A(t)$, i.e.
\[ A(t)v_i(t) = \sigma_i(t)u_i(t), \quad u_i(t)^*A(t) = \sigma_i(t)v_i(t)^*.\]
Using known formulas \cite{Gil2008}, we have 
\begin{align*}
\begin{split}
\frac{d\sigma_i(t)}{dt} &= \mathscr{Re}(u_i(t)^*(S_1A + AS_2)v_i(t))=\\
 &= \mathscr{Re}(u_i(t)^*S_1(\sigma_i(t)u_i(t) - t(S_1A + AS_2)v_i(t)) + (\sigma_i(t)v_i(t)^* - tu_i(t)^*(S_1A + AS_2))S_2v_i(t)) =\\
 &= \mathscr{Re}(- tu_i(t)^*S_1(S_1A + AS_2)v_i(t) - tu_i(t)^*(S_1A + AS_2)S_2v_i(t) =\\
 &= \mathscr{Re}(- tu_i(t)^*(S_1(S_1A + AS_2)+(S_1A + AS_2)S_2)v_i(t))).
\end{split}
\end{align*}
It follows that

\[ | \sigma(1) - \sigma(0)| \leq \frac{1}{2}\| S_1^2A + S_1AS_2 + S_2AS_1 +  AS_2^2 \|_2 \leq \frac12 \|E\|(\|S_1\|+\|S_2\|).\]

As for eigenvalues, using the entire $S_1$ and $S_2$ means including all the differences $\sigma_j^2 - \sigma_i^2$. Hence, if the gap between a certain couple of singular value is small, the bound can be loose.

Different bounds can be derived using previous bounds for eigenvalues.
Assume that the matrix $A \in \C^{m \times n}$, $m \geq n$ is of the form
\[ A = \begin{bmatrix}
    \Sigma \\
    0
\end{bmatrix}\]
where $\Sigma \in \R^{n \times n}$ is diagonal and contains the singular values.
The Jordan-Wielandt matrix

\[ M = \begin{bmatrix}
0 & A\\
A^* & 0
\end{bmatrix}\]
has eigenvalues equal to the singular values of $A$, plus some additional zeros.
This matrix is diagonalized by

\[ Q = \frac{1}{\sqrt2}\begin{bmatrix}
    I_n & I_n & 0\\
    0 & 0 & \sqrt2 I_{m-n}\\
    I_n & -I_n & 0
\end{bmatrix}.\]
Furthermore, if $A$ is perturbed as
\[ A +E = A + S_1A+AS_2\]
the matrix $M$ is perturbed as
\[ M + \begin{bmatrix}
    0 & E\\
    E^* & 0
\end{bmatrix}= \begin{bmatrix}
0 & A+E\\
(A+E)^* & 0
\end{bmatrix}.\]
Finally, if we multiply this matrix by $Q^*$ from the left and $Q$ from the right, we get
\[ Q^* \begin{bmatrix}
0 & A+E\\
(A+E)^* & 0
\end{bmatrix}Q = \begin{bmatrix}
    \Sigma + \frac{E_1+E_1^*}{2}&  \frac{-E_1+E_1^*}{2} & \frac{1}{\sqrt{2}}E_2^* \\
     \frac{E_1-E_1^*}{2} & -\Sigma- \frac{E_1+E_1^*}{2} & -\frac{1}{\sqrt{2}}E_2^* \\
    \frac{1}{\sqrt{2}}E_2 & -\frac{1}{\sqrt{2}}E_2 &0
\end{bmatrix}\]
where 
\[ E_1 =  S_1(1:n,1:n)\Sigma+\Sigma S_2, \quad E_2 = S_1(n+1:m, 1:n)\Sigma\]
are the first and last rows of $E$. We have reduced the original problem to the case of diagonal matrix plus off-diagonal perturbation.

Note that a block of zeros appears, which suggests that, if we only perturb the last rows of $A$ (the rows of zeros), then the perturbation of a singular value does not depend on the gap between the other singular values, but only on the gap with $0$, i.e. the singular value itself. This can be checked in the following way. Consider
\[ A +E = \begin{bmatrix}
    \Sigma \\ E_2
\end{bmatrix}\]
then
\[ (A+E)^*(A+E) = \Sigma^2 + E_2^*E_2.\]
If we call $\Tilde{\sigma}_i$ the $i$-th singular value of $A+E$, by Bauer-Fike we have
\[ |\tilde{\sigma}_i^2 - \sigma_i^2| \leq \|E_2^*E_2\| = \|E_2\|^2.\]
This implies that
\[ |\tilde{\sigma}_i - \sigma_i| \lesssim \frac{\|E_2\|^2}{\sigma_i} = \frac{\|E\|^2}{\sigma_i}.\]

We can compute the matrix $B$ that we have used in \Cref{sec:eigen} in this particular case. As $M$ is symmetric, $B$ will be skew-symmetric, and we can partition it as

\[ B = \begin{bmatrix}
    B_{1} & B_{2} & B_{3}\\
    B_{2}^* &  B_{1} & B_{3} \\
    B_3^* & B_3^* & 0
    
\end{bmatrix}\]
where
\[ B_{1}(i,j) = \frac{1}{2}\frac{E(i,j)+\bar{E}(j,i)}{\sigma_i - \sigma_j}\quad B_2(i,j) = \frac{1}{2}\frac{\bar{E}(j,i) - E(i,j)}{\sigma_i + \sigma_j} \quad B_3(i,j) = \frac{1}{\sqrt{2}}\frac{E(n+1+i,j)}{\sigma_i}.\]
These expressions agree with the asymptotic expansions. If we assume that $E$ is real, in the tangent and that $E_2 = 0$, then the asymptotic expansion is 

\begin{equation}\label{eq:asympsingular}
\sigma_i(D+E)
=
\sigma_i
+
\frac{1}{2\sigma_i}
\left[
\sum_{j\neq i} E(j,i)^2
+
\sum_{j\neq i}
\frac{(\sigma_iE(j, i)+\sigma_jE(i,j))^{2}}{\sigma_i^{2}-\sigma_j^{2}}
\right]
+
\mathcal{O}(\|E\|^3).
\end{equation}
If $E_1$ is symmetric, \Cref{eq:asympsingular} is equal to

\[  \frac{1}{2\sigma_i}
\left[
\sum_{j\neq i}
\frac{(\sigma_j+\sigma_i+\sigma_i-\sigma_j)\,E(i,j)^{2}}{\sigma_i-\sigma_j}
\right] = \sum_{j\neq i}
\frac{\,E(i,j)^{2}}{\sigma_i-\sigma_j}.\]
If $E_1$ is antisymmetric, \Cref{eq:asympsingular} is equal to
\[  \frac{1}{2\sigma_i}
\left[
\sum_{j\neq i}
\frac{(2\sigma_i^2 -2\sigma_i\sigma_j)E(i,j)^2}{\sigma_i^2-\sigma_j^2}
\right] = 
\sum_{j\neq i}
\frac{E(i,j)^2}{\sigma_i+\sigma_j}\]

\section{Jordan blocks}\label{sec:jordan}

Consider the nilpotent Jordan block

\[ J = \begin{bmatrix}
    0 & 1 & \\\
    & 0 & 1 &  \\
    & & 0 & 1 \\
    & & & \ddots & \ddots
\end{bmatrix} \in \C^{n \times n}\]
and a perturbation matrix $E$ of the form 
\[ E = JB- BJ\]
where $B \in \C^{n \times n}$.
Note that $E$ is such that the sum of the entries of each subdiagonal (including the main diagonal) is $0$. This, in particular, excludes the case of the entry in position $(n,1)$ being non-zero, which typically gives the result revealing the generic $\mathcal{O}(\varepsilon^{1/n})$ perturbation, which is the most sensitive behavior an eigenvalue of an $n\times n$ matrix can exhibit. Given a matrix $E$ that satisfies the condition of zero sum of subdiagonals, the matrix $B$ is unique up to an upper triangular Toeplitz matrix (again, $n^2 - n$ degrees of freedom).

For the Jordan block case, we give only asymptotic expansions for the eigenvalues.

 \begin{theorem}
      Let $J \in \C^{n \times n}$ the Jordan block of dimension $n \times n$ with eigenvalues equal to $0$. Let $E \in \C^{n \times n}$ be equal to
      \[ E = JB-BJ\]
      for some $B \in \C^{n \times n}$. Equivalently, $E$ has all subdiagonals summing to zero (including the main diagonal). 
      If 
      \[ \sum_{r=1}^{n-1}(-1)^rE(n,r+1)E(r,1)\neq 0\]
      then the perturbation of the eigenvalues of $J + \varepsilon E$
      has order $\mathcal{O}(\varepsilon^{2/n})$.
  \end{theorem}

  \begin{proof}
Define $p_{\varepsilon}(t)$ the characteristic polynomial of $J + \varepsilon E$. 
According to \cite[Thm 4]{Bur1992} applied to our case, in the expansion of $p_{\varepsilon}(t)$ in terms of $\varepsilon$, the linear term in $\varepsilon$ vanishes, that is
\[ p_{\varepsilon}(t) = t^n + \varepsilon^2 p^{(2)}(t) + \mathcal{O}(\varepsilon^3).\]
 In order to determine the magnitude of the perturbation of the eigenvalues, we have to determine whether the term $\varepsilon^2t^0$ in the polynomial $p_{\varepsilon}(t)$ vanishes or not. We know that the coefficient of $t^0$ in the characteristic polynomial is equal to
 \[ det(J + \varepsilon E).\]
 If we rewrite the determinant as
 \[ det(J + \varepsilon E)= \sum_{\sigma \in S_m}sgn(\sigma) \prod_{i = 1}^n (J(i,\sigma(i))+\varepsilon E(i,\sigma(i)))\]
 we see that, in order for the $\varepsilon^2$ to appear, we need that
 \[ J(i, \sigma(i)) = 1\]
 at least $n-2$ times. The only possibilities for $\sigma$ are the permutations in which there is an index $r$ such that 
 \[ \sigma(r) = 1\]
 \[ \sigma(n) = r+1\]
 while $\sigma(k) = k+1$ for all other indices.
 Finally, the coefficient of $\varepsilon^2t^0$ is
 \begin{equation}\label{eq:condjordan} \sum_{r=1}^{n-1}(-1)^rE(n,r+1)E(r,1).\end{equation}
  To conclude we use the technique of the Newton diagram \cite{Ma1998}, \cite[\S 8.3]{Bri2012}. In our case, the diagram is reported in \Cref{fig:newton}, where we see that the first exponent that leads to cancellation is $2/n$ if \Cref{eq:condjordan} is different from zero.
\end{proof}

 \begin{figure}\label{fig:newton}
\centering
 \begin{tikzpicture}[scale=0.8, every node/.style={font=\small}]

  \def\H{4}

  \def\nx{5}

  \draw[->] (-1,0) -- (\nx+1,0) node[right] {$x$};
  \draw[->] (0,-0.5) -- (0,\H+0.5) node[above] {$y$};

  \fill (0,2) circle (1.7pt) node[above left] {\tiny $(0,2)$};
  \fill (1,2) circle (1.7pt) node[above] {\tiny $(1,2)$};

  \node at (2,2) {$\cdots$};

  \fill (\nx-2,2) circle (1.7pt) node[above] {\tiny $(n-2,2)$};

  \fill (\nx,0) circle (1.7pt) node[below right] {\tiny $(n,0)$};

  \draw[very thick] (\nx,0) -- (0,2);

  \fill[red,opacity=0.3]
    (0,2)
    -- (\nx,0)
    -- (\nx-2,2)
    -- (\nx-4,\H)
    -- (0,\H)
    -- cycle;

  \draw[red!60!black,dashed]
    (\nx,0) --(\nx-2,2) --(\nx-4,\H) -- (0,\H);

\end{tikzpicture}
\caption{Netwon diagram in the generic case where the coefficient of $\varepsilon^2t^0$ is not zero. The slope of the solid line indicates the power of $\varepsilon$ in the perturbation of the eigenvalues ($2/n$).}
\end{figure}
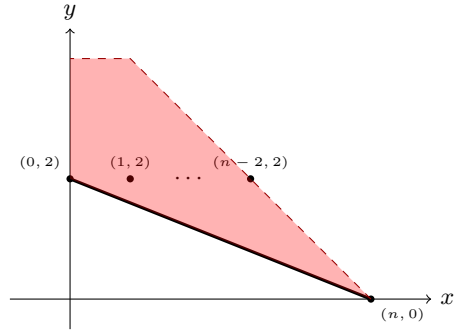

  Note that $2/n$ is the lowest coefficient that a line can have in the Newton diagram as the point $(n,0)$ is always in the diagram. As a consequence, $\mathcal{O}(\varepsilon^{2/n})$ is the largest perturbation an eigenvalue may have. In \Cref{fig:jordan} we plot the perturbed eigenvalues of nilpotent Jordan block.

\begin{figure}\label{fig:jordan}
\centering
\begin{tikzpicture}

\begin{axis}[
scale = 0.8,
    width=0.6\textwidth,
    height=0.6\textwidth,
    grid=both,
]

\addplot[mark = o, color=red, only marks] table[x=real, y = imag] {experiments/jordan.tex};

\addplot[
    domain=0:2*pi,
    samples=200,
    color = blue
]
({10^(-(3/50)) *cos(deg(x))},{10^(-(3/50)) * sin(deg(x))});

\end{axis}
\end{tikzpicture}
\caption{Perturbed eigenvalues for a nilpotent Jordan block of size $100 \times 100$ and a perturbation $E$ in the tangent with norm $\|E\| = 10^{-3}$. The circle has radius $\|E\|^{2/100}$.}
\end{figure}
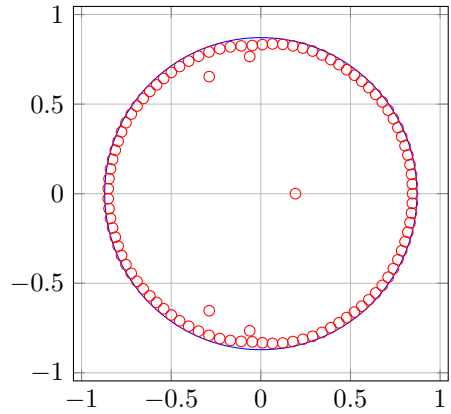

\section{Conclusion}

  We have analysed the perturbation of eigenvalues and eigenvectors of a matrix $A$ under an additive perturbation $E$ in the tangent space of isospectral matrices, that is $E = AB-BA$. We have shown how asymptotic expansions of eigenvectors and eigenvalues can be simply written in terms of $E$ and $B$ and we have highlighted how $B$ is closely related to the first order perturbation of eigenvectors. Beyond asymptotic expansions, we have provided high-order bounds on approximations of eigenvectors. Moreover, we have generalized the quadratic bound on eigenvalues proposed in \cite{Li2005} to this setting, which is a more general setting than Hermitian block off-diagonal perturbations. We have extended this analysis to the block-diagonal case, providing bounds on perturbed eigenspaces, which we further improved under the hypothesis of Hermiticity. Finally, we have extended the analysis to singular vectors, singular values and eigenvalues of Jordan blocks.

\bibliographystyle{siamplain}
\bibliography{biblio}
\end{document}